\documentclass[11pt]{amsart}
\usepackage[leqno]{amsmath}
\usepackage{amssymb}
\usepackage{amsthm}
\usepackage{fullpage}
\usepackage{etoolbox}
\usepackage{enumerate}   
\usepackage{bbm}   
\usepackage{mathtools}
\usepackage{tabularx} 
\usepackage[usenames,dvipsnames]{xcolor}
\usepackage[normalem]{ulem} 

\numberwithin{equation}{section}

\usepackage{subfiles} 
\usepackage[final]{hyperref} 
\hypersetup{
	colorlinks=true,       
	linkcolor=black,        
	citecolor=black,        
	filecolor=black,     
	urlcolor=black        
}

\input{commandsConformalRemovability.tex}

\newcommand{\SLE}{\mathrm{SLE}}

\title{On the Sobolev removability of the graph of one-dimensional Brownian motion}
\author{Cillian Doherty and Jason Miller}

\begin{document}
	\newpage
	\pagenumbering{arabic}

\maketitle

\begin{abstract}
Suppose that $B$ is a one-dimensional Brownian motion and let $\GG = \{ (t, B_t) : t \in [0,1]\}$ be the graph of $B|_{[0,1]}$. We characterize the Sobolev removability properties of $\GG$ by showing that $\GG$ is almost surely not $W^{1,p}$--removable for all $p \in [1, \infty)$ but is almost surely $W^{1,\infty}$--removable.
\end{abstract}
	
\setcounter{tocdepth}{1}
\tableofcontents

\parindent 0 pt
\setlength{\parskip}{0.20cm plus1mm minus1mm}

\section{Introduction}
\label{sec:intro}

\subsection{Definitions and background}
\label{subsec:def_and_background}

Suppose that $D \subseteq \C$ is open and $p \in [1,\infty]$.   We recall that the \emph{Sobolev space $W^{1,p}(D)$} consists of those functions $f \colon D \to \R$ such that $f \in L^p(D)$ and the derivative of $f$ exists in the weak sense and is in $L^p(D)$.  Suppose that $B$ is a standard one-dimensional Brownian motion and let $\GG = \{(t, B_t) \colon t \in [0,1]\} \subseteq \C$ be the graph of $B|_{[0,1]}$.   This paper is concerned with the \emph{Sobolev removability properties of $\GG$}, which we recall are defined as follows.

\begin{definition*}
    For $p \in [1,\infty]$ and a domain $D \subseteq \C$, a compact set $K \subseteq D$ is said to be \emph{$W^{1,p}$--removable inside $D$} if any continuous function $f \colon D \to \R$ which is in $W^{1,p}(D \setminus K)$ is in $W^{1,p}(D)$. 
\end{definition*}

We will in particular show that $\GG$ is a.s.\ not \emph{$\wp$--removable} for all $p \in [1,\infty)$ but is a.s.\ \emph{$\woi$--removable}, which completely characterizes when $\GG$ is Sobolev removable.  If $D_1$ and $D_2$ are both domains containing $K$ then $K$ is $W^{1,p}$--removable inside $D_1$ if and only if it is $\wp$--removable inside $D_2$ (see e.g. \cite{ntalampekos_removability_theorem}) meaning that $W^{1,p}$--removability is an intrinsic property of the set $K$ and does not depend on the domain $D$. Therefore, we are justified in saying that $K$ is simply $\wp$--removable, without reference to a domain $D$, if it is $\wp$--removable inside $\C$. The existence of a continuous function $f$ on some domain $D$ containing $K$, which is in $W^{1,p}(D\setminus K)$ but not in $W^{1,p}(D)$, certifies that $K$ is not $W^{1,p}$--removable. Such an $f$ is called an \emph{exceptional function}.

A related notion is that of \emph{conformal removability}, where a compact set $K \subseteq \C$ is said to be conformally removable if every homeomorphism $f\colon \C \to \C$ that is conformal on $\C\setminus K$ is conformal on all of $\C$.  If $1 \leq p < q \leq \infty$, any compact set $K$ that is $W^{1,p}$--removable is also $W^{1,q}$--removable, but the converse is false. It is also known that $W^{1,2}$--removability implies conformal removability; $W^{1,2}$--removability and conformal removability are in fact conjectured to be equivalent \cite{cned}, but this claim remains unproven. Due to examples we will mention in a moment, it is known that for any $p \in (2,\infty]$, $\wp$--removability does not imply conformal removability.

The $\wp$--removability and conformal removability of compact sets $K \subseteq \C$ is a much-studied topic. We only give a brief overview of past work here, and refer to some recent papers on the topic by Ntalampekos  \cite{ntalampekos_removability_theorem, ntalampekos_triangle} and the references therein for more details, as well as some short proofs of some of the claims made above.   We also refer to \cite{younsi} for a comprehensive survey of conformal removability\footnote{Note that our definition of ``conformally removable" corresponds to ``CH removable" in \cite{younsi}, and what is referred to there as conformal removability is a related but non-equivalent property.} and related topics.  It can be shown that any set of finite (or $\ss$-finite) $1$-Hausdorff measure is $\wp$--removable for all $p \in [1,\infty]$, and hence also conformally removable. In \cite{js}, Jones and Smirnov provided a sufficient condition for $W^{1,p}$--removability in the case where $K$ is the boundary of a simply connected domain, which is then applied to show that boundaries of H\"older domains are $W^{1,2}$--removable. The removability of random sets has also received attention.  For example, as explained in \cite{she2016zipper} the conformal removability of the $\SLE_\kappa$ curves for $\kappa \in (0,4)$ follows by combining the results of \cite{js} and \cite{rs2005basic} as the latter gives that the complementary components of such an $\SLE$ curve are H\"older domains.   Moreover, it was recently shown in \cite{ksle4, ksle_nonsimple} that the range of an $\text{SLE}_4$ curve is a.s.\ conformally removable, as is the range of $\text{SLE}_\kk$ for all $\kappa \in (4,8)$ such that the graph of complementary components of such a curve is connected.  The set of $\kappa \in (4,8)$ for which this property holds was shown to contain $(4,\kappa_0)$ for some $\kappa_0 > 4$ in \cite{gp2020connected}.  The results on the conformal removability of $\SLE_\kappa$ are motivated by its representation as a random conformal welding \cite{she2016zipper, dms2021mating}.

Any set $K$ with non-empty interior can be seen to be conformally non-removable and $\wp$ non-removable for any $p \in [1,\infty]$. More generally, any $K$ with non-zero area is neither conformally removable nor $\wp$--removable for any $p < \infty$, but this fails to hold for $p = \infty$ \cite{ntalampekos_triangle}. Letting $\cC$ denote the standard $1/3$--Cantor set, the product set $\cC \times [0,1]$ is compact, has zero area and is conformally non-removable and $\wp$ non-removable for all $p \in [1,\infty]$.  The starting point for the construction of the exceptional function in this case is the ``Devil's staircase'' function, which is a non-decreasing continuous function mapping $\cC$ to a set of positive Lebesgue measure.  The Sierpi\'nski carpet contains a copy of $\cC \times [0,1]$ so the same conclusions hold. The question of the removability properties of the Sierpi\'nski gasket is much more difficult, and was answered recently by Ntalampekos in \cite{ntalampekos_removability_theorem, ntalampekos_triangle}, who showed that it is not conformally removable, and therefore not $W^{1,p}$--removable for $p \leq 2$, but is $\wp$--removable for all $p \in (2, \infty]$.

\subsection{H\"older continuity and main results}
Recall that a function $g \colon [0,1] \to \R$ is said to be $\aal$-H\"older continuous with exponent $\aal \in (0,1]$ if there exists $C > 0$ such that $\n{g(x) - g(y)} \leq C\n{x - y}^\aal$ for all $x, y \in [0,1]$. The removability of graphs of functions has previously been investigated by Kaufman \cite{kaufman} and Tecu \cite{tecu}. Both constructed examples of graphs of $\aal$-H\"older functions for $\aal < 1/2$ that are neither $\woi$--removable nor conformally removable.  Tecu showed that if $g$ is $\aal$-H\"older continuous for $\aal > p/(2p-1)$ for $1 \leq p < \infty$, or if $\aal > 1/2$ for $p = \infty$, then its graph is $W^{1,p}$--removable, and conversely, if $\aal < p/(2p - 1)$ then there exist $\aal$-H\"older functions whose graphs are not $\wp$--removable. Setting $p=2$, we see that if $\aal > 2/3$, then the graph of any $\aal$-H\"older function must be conformally removable. For $1/2 \leq \aal \leq 2/3$, it remains an open question to answer whether there exist graphs of $\aal$-H\"older functions that are not conformally removable. Whether there exist $\aal$-H\"older graphs which are not $W^{1,p}$--removable for $\aal = p/(2p - 1)$ is also not known. It is a well-known result (see for example \cite{mp2010}) that a one-dimensional Brownian motion $B$ (viewed as a random function) is a.s.\ H\"older continuous on compact time intervals for all $\aal < 1/2$, but not for $\aal = 1/2$. 

\begin{theorem}
\label{thm:non_removable}
Let $B$ be a standard one-dimensional Brownian motion and let $\GG \subseteq \C$ be the graph of $B|_{[0,1]}$.   Then a.s.\ $\GG$ is not $W^{1,p}$--removable for all $p \in [1,\infty)$.
\end{theorem}

The proof of Theorem~\ref{thm:non_removable} is based on a ``stochastic version" of Tecu's construction \cite{tecu} to show the existence of graphs $\GG$ of $\aal$-H\"older continuous functions that are not $W^{1,p}$--removable, for $\aal < {p}/(2p - 1)$. In that case, the graph $\GG$ and the function $u \in W^{1,p}(D\setminus \GG) \setminus W^{1,p}(D)$ are constructed simultaneously using a version of the ``Devil's staircase'' function used to prove the non-removability of $\cC \times [0,1]$ and the difficulty arises in trying to ensure that $\GG$ is $\aal$-H\"older while keeping $u$ in $W^{1,p}$ off the graph.   In the setting of Theorem~\ref{thm:non_removable}, we are given the graph $\GG$ (that is a.s.\  H\"older continuous for all $\aal < 1/2$, but not for $\aal = 1/2$) and have to construct $u$ around it, which introduces some challenges.  We remark that if $W^{1,2}$--removability and conformal removability are shown to be equivalent, as has been conjectured, Theorem~\ref{thm:non_removable} will also guarantee that $\GG$ is a.s.\ not conformally removable.  Determining whether or not $\GG$ is conformally removable is currently an open question.  We remark that the behavior of $\GG$ under quasisymmetric mappings has been recently investigated in \cite{bhw2023browniangraph}.

\begin{theorem}
\label{thm:bdd}
Let $B$ be a standard one-dimensional Brownian motion and let $\GG \subseteq \C$ be the graph of $B|_{[0,1]}$.   Then a.s.\ $\GG$ is $W^{1,\infty}$--removable.
\end{theorem}

The argument used to prove Theorem~\ref{thm:bdd} is inspired by the proof of \cite[Theorem 6]{tecu}. The method used in \cite{tecu} (in the $\aal < 1/2$ case) to construct graphs which are not conformally removable immediately shows that these graphs are also not $W^{1,\infty}$--removable. Therefore, Theorem~\ref{thm:bdd} demonstrates that this method cannot be used to show that $\GG$ is not conformally removable without significant alteration.

\subsection*{Acknowledgements} C.D.\ was supported by EPSRC grant EP/W524633/1 and a studentship from Peterhouse, Cambridge.  J.M.\ was supported by ERC starting grant 804116 (SPRS).

\subsection*{Outline.} We will prove Theorem~\ref{thm:non_removable} in Section~\ref{sec:non_removability} and Theorem~\ref{thm:bdd} in Section~\ref{sec:bdd}.

\section{\texorpdfstring{Non-removability for $p < \infty$}{Non-removability}}
\label{sec:non_removability}

\subsection{Absolute continuity on lines}
The property of absolute continuity on lines is often central to questions on removability.  Given a domain $D \subseteq \C$, a function $f \colon D \to \R$ is said to be \textit{absolutely continuous on lines} (ACL) if the restriction of $f$ to (Lebesgue) a.e.\ horizontal and vertical line is absolutely continuous (here, $f$ is absolutely continuous on a line $\ell$ if it is absolutely continuous on any line segment of $\ell$ that lies entirely in $D$).

The following properties of absolute continuity on lines can be found, for example, in \cite[\S1.1.3]{mazya}. If a continuous function $f \colon D \to \R$ is in $W^{1,p}(D)$ for any $p \geq 1$ then it is ACL in $D$. Conversely, suppose $f$ is ACL in $D$.  Then on (Lebesgue) a.e.\ line parallel to the coordinate axes, $f$ is absolutely continuous and therefore has partial derivatives $\del_x f, \del_y f$ a.e. Moreover, if $f, \del_x f, \del_y f \in L^{p}(D)$, then $f \in W^{1,p}(D)$, and the distributional and classical derivatives of $f$ are a.e.\ equal. The outcome of these results is that if $f \in W^{1,p}(D \setminus K)$, $f$ is ACL on all of $D$
and $K$ has zero area (which holds if $K$ is a graph), then $f$ must be in $W^{1,p}(D)$.

\subsection{Proof of Theorem~\ref{thm:non_removable}}
In this section we will prove Theorem~\ref{thm:non_removable}, which states that the graph $\GG$ of a one-dimensional Brownian motion on $[0,1]$ is a.s.\ not $W^{1,p}$--removable for any $p \in [1, \infty)$.

The idea of the proof is to construct a nested family of closed sets $\GG_n$ where each $\GG_n$ is a finite union of rectangles in the plane in such a way that $\cap_n \GG_n \subseteq \GG$. Proposition~\ref{prop:r0good} shows that if the collection~$\GG_n$ satisfies a certain property (depending on $p \in [1,\infty)$), then we can construct a continuous function $u$ on a domain $D$ containing $\GG$ such that $u$ is in $\wp(D \setminus \cap_n \GG_n)$ but not in $\wp(D)$, proving that $\cap_n \GG_n$, and hence also $\GG$, is not $\wp$--removable. Lemmas~\ref{lem:good} and~\ref{lem:infgood}, combined with the proof of Theorem~\ref{thm:non_removable} at the end of this section, show that for any given $p \in [1,\infty)$, it is a.s.\ possible to choose the collection $\GG_n$ such that the property we have alluded to holds, from which we conclude that $\GG$ is a.s.\ not $\wp$--removable for any $p \in [1,\infty)$.

To state our lemmas, we redefine $B$ slightly to be a one-dimensional Brownian motion started from~$0$, stopped when it hits $1$, and conditioned to hit $1$ before $-1$, and we let $\GG$ be its graph.  We will show why this implies the theorem at the end of this section. For $K \in \N$ large and for each $n \in \N_0$ we define stopping times $T_0^n = 0$ and
\begin{equation}
\label{eq:tdef}
    T_{j+1}^n = \inf\left\{t > T_j^n \colon \n{B(t) - B(T^n_j)} = K^{-n}\right\} \quad\text{for}\quad j \geq 0.
\end{equation}

Let $\tau_1 = \inf\{t > 0 \colon B_t = 1\}$ and note that $\tau_1 = T_1^0$. To avoid pathologies, we assume that $\tau_1 < \infty$ which holds a.s. Suppose that $j, n \geq 0$ are such that $T_{j+1}^n \leq \tau_1$ and that there exists $\ell \in \Z$ with $B(T^n_j) = \ell K^{-n}$ and $B(T^n_{j+1}) = (\ell + 1)K^{-n}$ (note that we require that $B(T_{j+1}^n) - B(T_j^n)$ is \emph{positive}).
We define a \textit{box at scale $n$} to be a (closed) rectangle~$R$ with boundaries $x = T^n_j, T^n_{j+1}$ and $y = B(T^n_j), B(T^n_{j+1})$.  Conditional on $B(T^n_j)$ and $B(T^n_{j+1})$ (but not the stopping times $T^n_j, T^n_{j+1}$), the process $(B_{t} \colon T^n_j \leq t \leq  T^n_{j+1})$ will have the law of a Brownian motion started at $\ell K^{-n}$ that is conditioned to hit $(\ell + 1)K^{-n}$ before $(\ell - 1)K^{-n}$ and is stopped once this occurs. Note that this is simply a translation and rescaling of a standard Brownian motion started at 0 and conditioned to hit 1 before $-1$. The top and bottom boundaries of a box at scale~$n$ take values in a discrete set (multiples of $K^{-n}$) and we can partition the set of all such boxes into those lying at different heights. See Figure~\ref{fig:sim_boxes} for a depiction.

\begin{figure}[t]
\includegraphics[scale=0.8]{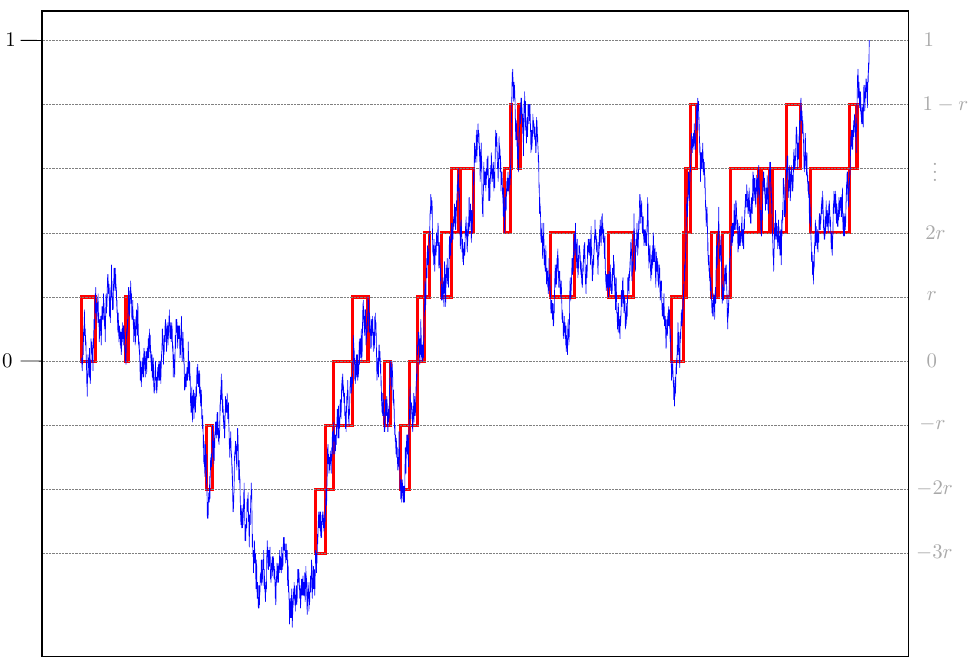}
\centering
\caption{A Brownian motion $B$ started from $0$ and conditioned to hit $1$ before $-1$. Also shown (in red) are the boxes at scale 1. Here we have set $K = 5$ and $r = \tfrac{1}{5}$. As in Definition~\ref{def:good}, $\cM = \{-3rK, -3rK + 1, \dots, K - rK - 1\}$, so the horizontal boundaries lie at heights $\{-3r, -3r + 1/K, -3r + 2/K, \ldots, 0, 1/K, \ldots, 1 - r\}$. The fact that $r = 1/K$ is not true in general, so we would usually have, for example, more possible box heights between $-3r$ and $-2r$. The law of $B$ inside any red box (if we do not condition on its width) is equal to a translation and scaling of the law of $B$.} 
\label{fig:sim_boxes}
\end{figure}

\begin{figure}[t]
\includegraphics[scale=1]{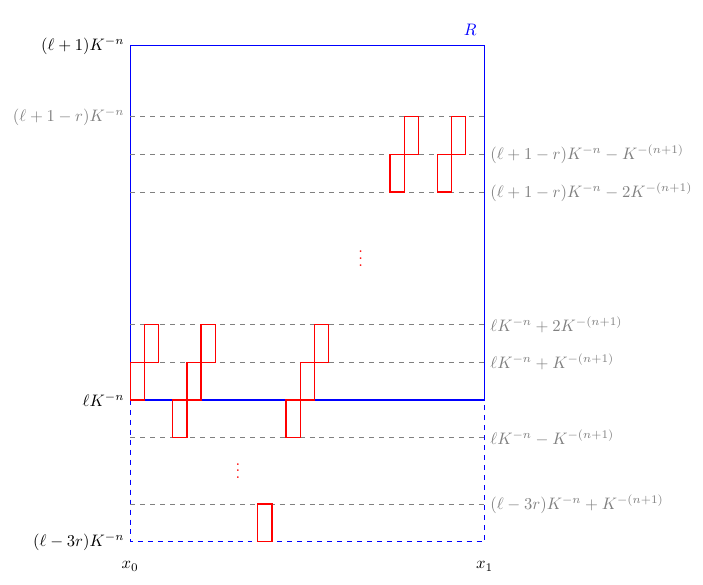}
\centering
\caption{The blue box $R$ (with solid boundary) is a box at scale $n$, with lower boundary at height $\ell K^{-n}$. The blue rectangle with dashed boundary shows the extension of $R$ to the space just below it, which we sometimes consider as part of~$R$. Each red box is a child of $R$, and is a box at scale $n+1$. $R$ is good if it has at least $cK$ children at each height (as indexed by $\cM$). Note that we have drawn this figure so that the small boxes all have the same width; this is merely for graphical purposes. In reality, each box will be stretched/compressed horizontally.
Note that we have not drawn the extensions of these boxes. If we do include these extensions, the red boxes may overlap vertically, but their interiors will not intersect. In particular, a horizontal line (as in the proof of Proposition~\ref{prop:r0good}) may pass through boxes at two different heights. (This figure is not drawn to scale.)}
\label{fig:blocks_within_block}
\end{figure}

\begin{definition}
\label{def:good}
Let $r,c > 0$ be small and chosen such that $rK \in \N$. We say that a box $R$ (as described above) is \emph{$(r,c)$-good}, or simply \emph{good}, if for each $m \in \cM \coloneqq \{-3rK, -3rK + 1, \dots, K - rK - 1\}$, the restriction of $B$ to $[T^n_j, T^n_{j+1}]$ makes at least $cK$ upcrossings of each of the intervals $[\ell K^{-n} + m K^{-(n+1)}, \ell K^{-n} + (m+1) K^{-(n+1)}]$. See Figure~\ref{fig:blocks_within_block} for further explanation.
\end{definition}

We assume $rK \in \N$ to avoid technicalities, since otherwise the set $\cM$ would not be well-defined as the sequence $-3rK, -3rK + 1, \dots$ may not include $K - rK - 1$. We will always be able to do so by choosing $r \in \Q$ and by making $K$ bigger, if necessary. In this case, at each of these $K + 2rK$ heights determined by $m \in \cM$, we will have at least $cK$ boxes of scale $n+1$ that are ``contained" in~$R$ (although strictly they can be slightly below the lower boundary of~$R$). We refer to such boxes as \textit{children} of $R$. See Figure~\ref{fig:blocks_within_block} for a visualization.

By Brownian scaling and the translation invariance of Brownian motion, the probability that a given box is good is a constant depending only on $K$, $c$, and $r$, and does not depend on $n$, $j$, or $\ell$ (that is, the location or scale of the box). Furthermore, whether a fixed box is good is independent of whether any other box of the same scale, or any earlier scale, is good. Due to these considerations, the probability of the event described in the following lemma is the probability that any such box is $(r,c)$-good. It is worth mentioning that we are going to allow $K$ to change, so a box (at scale $n$) could perhaps be referred to as a ``$K$-box (at scale $n$)" since the definition of a box depends on $K$. We will not do this here however, since $K$ should always be clear from the context, but it is worth keeping this in mind in the following lemmas.

\begin{lemma}
\label{lem:good}
For every $q_0 \in (0,1)$ there exist $r, c> 0$ and $K_0 = K_0(c) \in \N$ such that, for all $K \geq K_0$ (with $rK \in \N$), the probability $q \equiv q_{r,c,K}$ that a Brownian motion $B$ started from $0$ and conditioned to hit $1$ before $-1$ has at least $cK$ upcrossings of every interval of the form $[mK^{-1}, (m+1)K\nv]$ for $m \in \cM$ is at least $q_0$.
\end{lemma}
\begin{proof}
Fix $q_0  \in (0,1)$ and set $\ee = 1 - q_0$. Let $W$ be a standard Brownian motion started from 0 and let $L_x^W(t)$ denote its local time (as defined in \cite{mp2010}) at a point $x \in \R$ at time $t$. In the following, for $a \in \R$, let $\tau_a^W = \inf\{t \geq 0\colon W_t = a\}$. Choose $r > 0$ such that
\begin{equation}\label{eq:hit3r}
    \P[ \tau_1^W < \tau_{-3r}^W ] < \frac{\ee}{6}.
\end{equation}

Let $V$ be a Brownian motion started from $-3r$ with local time $L^V_x(t)$ and set $\tau^V_1 = \inf\{t \geq 0 \colon V_t = 1\}$. By the first Ray--Knight theorem \cite[Theorem 6.28]{mp2010}, 
\[ \P\!\left[ \inf_{x \in [-3r, 1-r]} L_x^V(\tau_1^V) = 0 \right] = 0.\]
This ensures that there exists $\dd > 0$ such that
\begin{equation}
\label{eq:ray_knight}
    \P\!\left[ \inf_{ x \in [-3r, 1 - r]} L^V_x(\tau^V_1) < \dd\right] < \frac\ee6.
\end{equation}
Noting that $(V_t)_{t \geq 0}$ and $(W_{\tau_{-3r}^W + t})_{t\geq 0}$ have the same distribution, we can combine~\eqref{eq:hit3r} and~\eqref{eq:ray_knight} to conclude that
\begin{equation}\label{eq:small_local_time}
     \P\!\left[ \inf_{x \in [-3r, 1 - r]} L^W_x(\tau_1^W) < \dd\right] < \frac\ee3.
\end{equation}

By \cite[Theorem 6.18]{mp2010}, we have for each interval $I \subseteq \R$ that
\[ \int_0^{\tau_1^W}\one_{W_s \in I} ds = \int_I L_x^W(\tau_1^W)dx.\]
It follows that if the event in~\eqref{eq:small_local_time} does not hold (which happens with probability close to $1$) then $W$ spends time at least $\dd/K$ in the interval $[m K\nv, (m + 1)K\nv]$ before $\tau_1^W$ for each $m \in \cM$. We want to show that, for small $c > 0$, the event that $W$ has $cK$ upcrossings of each of these intervals before $\tau_1^W$ has a high probability, and in fact has a high probability uniformly in $K$.


For each $a \geq 0$ define $\ss_a^W = \inf\{t \geq 0\colon \n{W_t} = a\}$. Then there exists $\la > 0$ such that $\psi(\la) = \log \E(e^{\la \ss_1}) < \infty$.  Fix $\la$ and choose $c > 0$ small enough that $4c\psi(\la) < \dd \la$. The time spent in the interval $[0,a]$ during the first upcrossing (and also all subsequent upcrossings) of $[0,a]$ has distribution (see \cite[Theorem 5.38]{mp2010})
\begin{equation}
    \int_0^{\tau_a^W}\one(W_t \in [0,a])dt \:\stackrel{d}{=}\: \ss_a^W.
\end{equation}

Consider now the first $J = 2\lceil{cK}\rceil + 1$ crossings (either up- or downcrossings) of $[0,1/K]$. Then $J \leq 4cK$ for sufficiently large $K$ (for fixed $c$). This set of crossings will include at least $\lceil{cK}\rceil$ upcrossings. We bound the probability that $W$ spends too long in $[0,1/K]$ before enough upcrossings have occurred. The probability that $W$ spends more than time $\dd K$ in this interval by the time $J$ crossings have been completed is equal to the probability that the sum of $J$ independent copies of $\ss_{1/K}^W$ (denote these by $\ss_{1/K}^i$) is greater than $\dd/K$. Let $\ss_1^{i}$ denote i.i.d.\ copies of $\ss_1^W$. We have that
\begin{flalign*}
    & & \P\!\left[\sum_{i=1}^J \ss_{1/K}^{i} > \frac\dd K \right] &= \P\!\left[ \sum_{i=1}^J \ss_1^{i}  > \dd K\right] &\text{(since $\ss_a^W \stackrel{d}{=} a^2 \ss_1^W$ by scaling)}\\
    & & &\leq \E\!\left[\exp\left(\la \sum_{i=1}^J \ss_1^{i}\right)\right] e^{-\la \dd K} &\text{(by a Chernoff bound)}\\
    & & &\leq \exp(\psi(\la))^J e^{-\la \dd K} &\text{(independence)}\\
    & & & \leq \exp[-K(\dd \la - 4c\psi(\la))]. &
\end{flalign*}
By our choice of $c$, the exponent in the last line is negative. Note that this same argument applies to all intervals $I_m = [m K\nv, (m + 1)K\nv]$ for $m \in \cM$. Let $A_K$ be the event that, for all $m \in \cM$, $W$ makes $\lceil{cK}\rceil$ upcrossings of $I_m$ before it spends $\dd/K$ units of time in $I_m$. By a union bound,
\begin{equation}
\label{eq:akc}
    \P[ A_K^c ] \leq (1 + 2r)K\exp(-K(\dd \la - 4c\psi(\la))) \leq \frac{\ee}{6},
\end{equation}
where the last inequality holds for all $K$ greater than or equal to some constant $K_0$ depending on $c$ since the middle expression goes to 0 as $K \to \infty$.

We can now conclude the proof. Note that the event $A_K$ is defined for $W$ even if some of the upcrossings happen after $\tau_1^W$, since we have defined this event without any reference to $\tau_1^W$. Suppose it holds that $L^W_x(\tau_1^W) \geq \dd$ for all $x \in [-3r, 1 -r]$ and also that $A_K$ holds. Then $W$ must make at least $cK$ upcrossings of each interval $I_m$ for $m \in \cM$ before $\tau_1^W$. Call this event $E$. By~\eqref{eq:small_local_time} and~\eqref{eq:akc} we have that
\[ \P[E^c] \leq \P\!\left[  \inf_{x \in [-3r, 1 - r]} L^W_x(\tau^W_1) < \dd\right] + \P[ A_K^c] \leq \frac{\ee}{2}.\]
Recall that $W$ is a standard Brownian motion started from $0$, but we want a corresponding result for $B$, which is a Brownian motion conditioned to hit 1 before $-1$. But we can easily extend our result to this case:
\[ \P[E^c \mid \tau_1^W < \tau_{-1}^W] = \frac{\P[E^c, \tau_1^W < \tau_{-1}^W]}{\P[\tau_1^W < \tau_{-1}^W]} \leq \frac{\P[E^c]}{1/2} \leq \ee.\]
It follows that for some choice of $r$, $c$, and $K_0$, for all $K \geq K_0$, the probability that $B$ makes at least~$cK$ upcrossings of each interval $I_m$ for $m \in \cM$ before hitting $1$ is at least $q_0$.
\end{proof}

Fix $q_0 \in (0,1)$, choose $r,c > 0$ as in Lemma~\ref{lem:good} and define $s = c/2$. Suppose $K$ is sufficiently large that the claim in the lemma holds. We say a box $R$ is $\cG_1$ (or $1$-good) if it is $(r, s)$-good, which is a strictly weaker condition that being $(r, c)$-good. In particular, we emphasize that being \emph{$1$-good} and being \emph{good} are not equivalent. For $n \geq 1$, we say $R$ is $\cG_{n+1}$ or $(n+1)$-good if it has at least $sK$ children at each of the $K + 2rK$ heights that are $n$-good (see Figure~\ref{fig:blocks_within_block}). Say a box is in $\cG_\infty$ if it is in $\cG_n$ for every $n$. Then a box in $\cG_\infty$ will have at least $sK$ children in $\cG_\infty$ at each of the $K + 2rK$ possible heights, thus producing an infinite nested structure of boxes all in $\cG_\infty$. Note that a box being $\cG_{n+1}$ means it must also be $\cG_n$.

We fix the \emph{box at scale 0} to be $R_0 = [0, T^0_1] \times [0,1]$. Recall that we assumed that $T^0_1$ is finite so that this is a well-define box. Recall also that the property of being good (or of being $\cG_\infty$) depends on $r$, $c$, and~$K$, and that a given box is good with probability $q = q_{r,c,K}$. The next lemma shows that there is a non-zero probability that $R_0$ is in $\cG_\infty$ for some $r$, $c$, and $K$. As before, by scale and translation invariance, the probability that a given box is $\cG_\infty$ depends only on $r$, $c$, and $K$, and not its location or scale, and furthermore this event is independent of any other box at the same scale being $\cG_\infty$ (and indeed at earlier scales as long as they are not ancestors of the box in question).

\begin{lemma}
\label{lem:infgood}
There exist $r,c > 0$ and $K_0 = K_0(c) \in \N$ such that for all $K \geq K_0$ (with $rK \in \N$) we have that
\[ \P[ R_0 \in \cG_\infty ] > 0.\]
\end{lemma}
\begin{proof}
Using Lemma~\ref{lem:good}, choose $r,c > 0$ and $K'_0 \in \N$ such that for all $K \geq K'_0$ we have that $q = q_{r,c,K}$ is greater than or equal to $q_0 \coloneqq 7/8$. Fix one such $K$. Let $A = (A_m)_{m \in \cM} \in \N_0^{K + 2rK}$ be the random vector representing the number of children of a given box at each height. Let $R_0$ be the box at scale $0$. Let $\Bin(n,p)$ denote the binomial distribution with parameters $n$ and $p$ and set $\aal_1 = \P[R_0 \in \cG_1]$.  For each $n \geq 2$, we also set
\begin{align*}
    \aal_n := \P[R_0 \in \cG_n] &= \sum_{\vec{a} \in \N_0^{K + rK}} \P[A = \vec{a}] \P[\text{$\forall j \in \cM$, at least $sK$ out of $a_j$ i.i.d.\ copies of $R_0$ are $\cG_{n-1}$}]\\
    &= \sum_{\vec{a}} \P[A = \vec{a}]\prod_{j \in \cM}\P[ \Bin(a_j,\aal_{n-1}) \geq sK].
\end{align*}
Motivated by this we define
\begin{equation}
\label{eq:fdef}
    F(t) = \sum_{\vec{a}} \P[A = \vec{a}]\prod_{j\in\cM}\P[\Bin(a_j,t) \geq sK],
\end{equation}
so that $\aal_{n+1} = F(\aal_n)$ for all $n$.
Now, $\P[A_j \geq cK, \;\forall j] \geq q_0$ by assumption, and clearly the probabilities in the product in~\eqref{eq:fdef} are increasing in $a_j$. Therefore\footnote{If $cK, sK$ are not integers, we can replace them by $\lceil{cK}\rceil, \lceil{sK}\rceil$, respectively, and make minor changes to the argument to get the same conclusion. This technicality could also be avoided by just assuming $cK, sK$ are integers, which is possible since we can always take $c$ to be rational and $K$ to be larger.}
\begin{align*}
F(t) \geq q_0\, \P[ \Bin(cK, t) \geq sK]^{K(1 + 2r)} &= q_0\left(1 - \P[ \Bin(cK, t) < sK]\right)^{K(1 + 2r)}\\
&\geq q_0\left(1 - K(1 + 2r)\P[ \Bin(cK, t) < sK]\right).
\end{align*}
Next we bound this last probability using a Chernoff bound to see for $t > 3/4$ that (recall $s = c/2$)
\begin{align*}
    \P[ \Bin(cK, t) < sK] &\leq \exp\left(\tfrac12{cK}\left[\log(2(1-t)) + \log(2t)\right]\right)\\
    &=  \exp\left(\tfrac12{cK}\log(4t(1-t))\right)\\
    &\leq \exp\left(\tfrac12{cK}\log\tfrac34\right).
\end{align*}
Inserting this into the previous inequality we have for $t > 3/4$ that
\begin{equation}
    F(t) \geq q_0\left(1 - K(1+2r)\exp\left(\tfrac12{cK}\log\tfrac34\right)\right).
\end{equation}
Therefore since we chose $q_0 = 7/8$, there exists $K_0$ (chosen so that $K_0 \geq K_0'$) such that if $K \geq K_0$ then $F(t) > 3/4$ whenever $t > 3/4$. Since $\aal_1 = \P[R_0 \in \cG_1] = q_0 > 3/4$, we have that $\aal_n > 3/4$ for all $n$, allowing us to conclude that $\P[R_0 \in \cG_\infty] = \lnti \P[R_0 \in \cG_n] \geq 3/4$.
\end{proof}

Finally we are in a position to address the question of removability. We do so by showing that if~$R_0$ is~$\cG_\infty$ for some choice of $r$, $c$, $K$ with $K > s^{-p}$, then~$\GG$ is not~$W^{1,p}$--removable. Once we have that $R_0 \in \cG_\infty$ we know we have an infinite nested family of $\cG_\infty$ boxes which we use to construct the exceptional function $f$. Let $\GG_n$ denote the union of all the boxes in this family at scale $n$. Then $\GG_{n+1} \subseteq \GG_n$ for each $n$ and $\cap_n \GG_n \subseteq \GG$. The following proposition (if we occasionally replace $\GG$ by $\cap_n \GG_n$) in fact shows that $\cap_n \GG_n$ is not removable, meaning it can be viewed essentially as a deterministic result on families of nested boxes satisfying certain properties.

\begin{prop}\label{prop:r0good}
Fix $p \in [1, \infty)$. If $R_0 \in \cG_\infty$ for $r,c,K$ with $K > s^{-p}$ (where $s = c/2$) then~$\GG$ is not~$W^{1,p}$--removable.
\end{prop}
\begin{proof}
Let $D$ be a finite open square containing $\GG$, $R_0$, and all of the descendants of $R_0$. We will construct a function $u$ in $W^{1,p}(D \setminus \GG)$ which is not ACL on all of $D$.

It will be helpful from now on to redefine our boxes slightly, so that they include the area immediately below them. Specifically, attach to a box $R_n$ at scale $n$ the rectangle directly below this box of height $3rK^{-n}$ with the same width (this is the dashed blue box in Figure~\ref{fig:blocks_within_block}). From now on~$R_n$ will be assumed to mean this slightly larger box and will have height $(1+3r)K^{-n}$. Such a box can then have children at the $(1+2r)K$ different possible heights indexed by $\cM$ as in Definition \ref{def:good} (this is not $(1 + 3r)K$ since we do not look for any children too close to the top of the box). Suppose $R_0 \in \cG_\infty$ and fix an infinite family $\cF$ of nested~$\cG_\infty$ boxes so that every box in $\cF$ has exactly $sK$ children in $\cF$ at each of these possible heights.

Note that any point $(x,y) \in \R^2$ can be in the interior of at most one box at scale $n$. If such a box exists, call it $R_n(x,y)$ and let $w_n(x,y)$ be its width. If $(x,y)$ is in two boxes at scale $n$ (which can only occur when $(x,y)$ lies on the right-hand boundary of one box and the left-hand boundary of another) we define $R_n(x,y)$ and $w_n(x,y)$ to correspond to the rightmost of these two boxes. The height of $R_n(x,y)$ must be $(1+3r)K^{-n}$. Recall that $\GG_n$ denotes the closure of the union of the boxes in $\cF$ at scale $n$. To avoid some technicalities, it will be helpful to define $w_n(x,y) = w_n(x,y')$ in the case that $(x,y) \notin \GG_n$ but $(x,y') \in \GG_n$ for some $y'$ (notice that every vertical line passes through at most one box at any fixed scale).

Let $g$ be smooth, zero outside $[\tfrac14, \tfrac34]$, non-negative, and with $\inn{g} = 1$.  Set $A_0(x,y) = g(y)/w_0(x,y)$ for $x \in [0,T_1^0]$ and define $A_0(x,y) = 0$ for $x$ elsewhere (note that $w_0(x,y) = T_1^0$ for $x \in [0,T_1^0]$). Let $\phi \colon \R \to [0,1]$ be a smooth function that is 1 on $[-2r,1 - 3r]$, 0 outside $[-3r, 1-2r]$ and satisfies $\phi(x - 1) + \phi(x) = 1$ for all $x \in [0,1]$. For a rectangle $R$ at scale $n$ with lower boundary $(\ell - 3r)K^{-n}$ and upper boundary $(\ell + 1)K^{-n}$, we define 
\begin{equation}
    \phi^R(y) = \phi(K^n y - \ell).
\end{equation}
Then $\phi^R$ is smooth and supported on the projection of $R$ onto the $y$-axis. The functions $\phi^R$ will act as a kind of partition of unity. If $R_n(x,y)$ is not defined for some $(x,y) \notin \GG_n$, we set $\phi^{R_n(x,y)} \equiv 0$ on $\R$.

We define $A_{n}(x,y)$ inductively. Suppose $A_n$ has been determined. If $(x,y) \notin \GG_{n+1}$, set $A_{n+1}(x,y) = 0$. Otherwise, define
\begin{equation}
    A_{n+1}(x,y) = \frac{1}{sK}\frac{w_n(x,y)}{w_{n+1}(x,y)} A_n(x,y) \phi^{R_{n + 1}(x,y)}(y).
\end{equation}
Define $\aal_n(x,y) = w_n(x,y) A_n(x,y)$. Note that $\aal_0(x,y) \leq 1$ everywhere by the definition of $A_0$. Inductively we have
\begin{equation}
\label{eqn:alpha_bound}
\aal_n \leq (sK)^{-n}
\end{equation}
everywhere and that $A_n$ and~$\aal_n$ are constant in~$x$ within boxes in~$\GG_n$.

Suppose $(x,y) \in \GG_n$, and that $x_0$ and $x_1$ are the left and right boundaries of $R = R_n(x,y)$ respectively as in Figure~\ref{fig:blocks_within_block}. Consider the horizontal line segment $L$ at height $y$ in $R_n(x,y)$. We have that
\begin{equation}
\label{eqn:a_n_alpha}
\int_{x_0}^{x_1} A_{n}(t,y)dt = A_n(x,y)w_n(x,y) = \aal_n(x,y).
\end{equation}
Now consider the boxes at scale $n+1$ in $R$ as in Figure~\ref{fig:blocks_within_block}. As we integrate $A_{n+1}$ over $L$, each time we pass through a box, containing a point $(x^*, y)$ say, we pick up a contribution of $\aal_{n+1}(x^*,y)$. If $(\ell - 3r)K^{-n} < y < (\ell + 1 - 2r) K^{-n}$ for some $\ell \in \Z$, then there are two cases. First suppose~$L$ passes through boxes at scale $n+1$ at only a single height. If $R'$ is one of these boxes we have $\phi^{R'}(y) = 1$. Otherwise $L$ passes through boxes of two different heights. In this case, if $R^1, R^2$ are representatives of boxes at these two heights, we have $\phi^{R^1}(y) + \phi^{R^2}(y) = 1$. In either case, we pass through exactly $sK$ boxes of each height, and we get, from the definition of $\aal_{n+1}$,
\[ \int_{x_0}^{x_1} A_{n+1}(t,y)dt = \aal_n(x,y).\]
If $y \notin ((\ell - 3r)K^{-n},  (\ell + 1 - 2r)K^{-n})$, then we have $A_n(x,y) = 0$ on $R$ since $\phi^R(y) = 0$. In either case we get
\begin{equation}
    \int_{x_0}^{x_1} A_{n}(t,y)dt = \int_{x_0}^{x_1} A_{n+1}(t,y)dt.
\end{equation}

Define
\begin{equation}
\label{eqn:u_n_def}
    u_n(x,y) = \int_0^x A_n(t,y) dt.
\end{equation}
By the previous equation we have that $u_n = u_{n+1}$ off $\GG_n$. For a given point $(x,y) \in \GG_n$ let $(x_0, y)$ be a point on the left boundary of $R_n(x,y)$ and $(x_1, y)$ a point on the right. Since $u_n$ is non-decreasing in $x$, we have
\begin{flalign*}
& & \n{u_{n+1}(x,y) - u_n(x,y)}
&\leq \n{u_n(x_0, y) - u_n(x_1, y)} &\text{(since $u_{n+1}(x_0,y) = u_n(x_0,y)$)}\\
& & &\leq A_n(x,y)w_n(x,y) &\text{(by~\eqref{eqn:a_n_alpha} and~\eqref{eqn:u_n_def})}\\
& & &= \aal_n(x,y) &\text{(by~\eqref{eqn:a_n_alpha})} \\
& & &\leq (sK)^{-n} &\text{(by~\eqref{eqn:alpha_bound})}.
\end{flalign*}
It follows that $\inn{u_{n+1} - u_n} \leq \aal_n \leq (sK)^{-n}$ so that the $u_n$ converge uniformly to a continuous function $u$. We see that $u = u_n$ off $\GG_n$, so in particular $\del_x u = 0$ off $\GG_n$ for all $n$, and hence off $\GG$. Therefore, on a fixed horizontal line, we have that $\del_x u = 0$ a.e., but $u(0,y) < u(T_1^0, y)$ for all $y$ in some open set (by the definition of $g$ and the construction of $u$). This means that $x \mapsto u(x,y)$ is not absolutely continuous on these lines, so $u$ is not ACL on $D$. Similarly, since $u = u_n$ off $\GG_n$, we can show that $\del_y u$ is continuous in a small neighbourhood of any point not in $\GG$ (since it will be constant in $x$). Therefore $u$ is differentiable on $D\setminus \GG$, so showing $\del_y u \in L^p(D \setminus \GG)$ will ensure that $u \in W^{1,p}(D \setminus \GG)$.

Set
\[ e_n = \sup_{x,y} \n{\del_y \aal_n(x,y)}\]
(for fixed $x$, $\aal_n$ is smooth as a function of $y$). Then we get from the definition of $\aal_{n+1}$ that
\begin{align*}
    \frac{\del \aal_{n+1}(x,y)}{\del y} &= \frac{1}{sK}\left(\frac{\del \aal_{n}(x,y)}{\del y} \phi^{R_{n+1}(x,y)}(y) + \aal_n(x,y)\frac{\del \phi^{R_{n+1}(x,y)}(y)}{\del y} \right)\\
    \implies e_{n+1} &\leq \frac{1}{sK}\left(e_n + (sK)^{-n}{c_0} K^{n+1}\right)\\
    \implies e_{n+1} &\leq \frac{e_n}{sK} + c_0 s^{-(n+1)}.
\end{align*}
Here, $c_0$ is some constant that bounds the derivative of $\phi$ and we use that $\phi^{R_{n+1}(x,y)}$ has derivative proportional to $K^{n+1}$ since it is a scaling (and translation) of $\phi$ if $(x,y) \in \GG_{n+1}$. If $(x,y) \notin \GG_{n+1}$ then $\phi^{R_{n+1}(x,y)}$ is the zero function and the inequality holds trivially. We can conclude that 
\begin{equation}\label{eq:en_bound}
    e_n \lesssim s^{-n}.
\end{equation}

Now define
\[ v_n = \sup_{x,y} \n{\frac{\del u_n(x,y)}{\del y}}.\]
Let $(x,y_1), (x,y_2)$ be two points in a rectangle at scale $n$. Let $x'$ be the left boundary of this rectangle, so that $u_n = u_{n+1}$ here. Then we have that
\begin{align}
    \n{u_{n+1}(x,y_1) - u_{n+1}(x,y_2)} &\leq  \n{u_{n+1}(x',y_1) - u_{n+1}(x',y_2)} + \n{\int_{x'}^x (A_{n+1}(t,y_1) - A_{n+1}(t,y_2))dt} \notag\\
    &= \n{u_{n}(x',y_1) - u_{n}(x',y_2)} + \n{\int_{x'}^x (A_{n+1}(t,y_1) - A_{n+1}(t,y_2))dt} \notag\\
    &\leq v_n \n{y_2 - y_1} + \n{\int_{x'}^x (A_{n+1}(t,y_1) - A_{n+1}(t,y_2))dt} \notag\\
    &\leq \n{y_2 - y_1} (v_n +  s(1 + 2r)K^2 e_{n+1}) \quad\text{(see below for explanation)}. \label{eqn:un_y_change}
\end{align}
The last line comes from the following. Suppose we index the red boxes in Figure~\ref{fig:blocks_within_block} by a finite set~$I$, and for $i \in I$, let a box have left boundary at $a_i$ and right boundary at $b_i$. Let $t_i \in (a_i, b_i)$ and note that $A_n(w, y) = A_n(t_i, y)$ for all $w \in (a_i, b_i)$ since $A_{n+1}$ is constant in $x$ within boxes. Note that $\n{I} = sK \cdot (1 + 2r)K$. Then
\begin{align*}
    \n{\int_{x'}^x (A_{n+1}(t,y_1) - A_{n+1}(t,y_2))dt} &= \sum_{i \in I} \n{\int_{a_i}^{b_i} (A_{n+1}(t,y_1) - A_{n+1}(t,y_2))dt}\\
    &= \sum_{i \in I} \n{\aal_{n+1}(t_i, y_1) - \aal_{n+1}(t_i, y_2)}\\
    &\leq \sum_{i \in I} \n{y_2 - y_1}e_{n+1}\\
    &= \n{y_2 - y_1}\n{I}e_{n+1}.    
\end{align*}
Combining~\eqref{eqn:un_y_change} with~\eqref{eq:en_bound} we conclude that
\begin{equation}
    v_{n+1} \leq v_n + C's^{-(n+1)},
\end{equation}
where $C'$ depends on $r,s$ and $K$ but is constant in $n$.
We can conclude that
\begin{equation}
\label{eqn:v_n_bound}
    v_n \lesssim s^{-n}.
\end{equation}
Since $u = u_n$ off $\GG_n$, we have that $\n{\nabla u} = \n{\nabla u_n} \leq v_n$ on $\GG_{n-1} \setminus \GG_n$ for $n \geq 1$. Denote the area of $R_0 = \GG_0$ by $A$, which we have assumed is finite. Note that $\n{\nabla u_0}$ is bounded by some constant depending on $g$, which we use to deal with the integral over $D \setminus \GG_0$ below (combined with the fact that $D$ is bounded). By construction, the area of $\GG_n$ is at most $(1+3r)K^{-n}A$, so
\begin{flalign*}
    & & \int_{D \setminus \GG} \n{\nabla u}^p dx dy &\leq \int_{D\setminus \GG_0}\n{\nabla u}^p dx dy + \sum_{n = 0}^\infty \int_{\GG_{n}\setminus \GG_{n+1}} \n{\nabla u}^p dx dy\\ 
    & & &\leq \wt{C} + \sum_{n=0}^\infty (1+3r)K^{-n} A\,v_{n+1}^{p} &\hspace{-40pt}\text{($\wt{C}$ is a constant depending on $D$ and $g$)}\\
    & & &\lesssim \sum_{n=0}^\infty K^{-n}s^{-(n+1)p}  &\hspace{-40pt}\text{(by~\eqref{eqn:v_n_bound})}\\
    & & &\lesssim \sum_{n=0}^\infty (s^{-p} K^{-1})^n. &
\end{flalign*}
This sum is finite by our assumption that $s^{-p} < K$. Therefore $u \in W^{1,p}(D \setminus \GG)$ but is not ACL, and hence not in $W^{1,p}(D)$.
\end{proof}

Finally, we can conclude the proof of our initial statement. The above results show that if $B$ is a Brownian motion conditioned to hit 1 before $-1$ and then stopped, there is a positive probability that its graph is not removable for $W^{1,p}$. We want to show that if $\wt{B} = (\wt{B}_t)_{0 \leq t \leq 1}$ is the original standard Brownian motion of the theorem statement, then its graph $\GG$ is \emph{almost surely} not removable. The idea of the proof is to show that $\wt B$ contains many independent copies of processes whose law is a rescaling of $B$, each of which is non-removable with a certain probability, independently of the others, therefore showing that $\wt B$ is removable with large probability. We formalize this argument in the following proof.

\begin{proof}[Proof of Theorem~\ref{thm:non_removable}]
We have shown that if $B$ is a Brownian motion conditioned to hit 1 before $-1$ and then stopped, there is a positive probability $q > 0$ that its graph is not removable for $W^{1,p}$ (we fix $p$ here and just say ``removable" from now on). Since we can apply the same process to the reflection of $B$, it follows that the same holds for a Brownian motion $B'$ started at $0$ and stopped whenever it hits $1$ \emph{or} $-1$. Removability is also unaffected by scaling, so we can replace $\pm 1$ by $\pm a$ for any $a > 0$.

Let $W = (W_t)_{t \geq 0}$ be a one-dimensional Brownian motion started from 0 with graph $\GG$. Fix $a > 0$, set $S_0 = 0$ and, similarly to~\eqref{eq:tdef},  for $j \geq 0$ set
\begin{equation}
    S_{j+1} = \inf\left\{t > S_j \colon \n{W_t - W_{S_j}} = a\right\}.
\end{equation}
Now each $(W_t)_{S_j \leq t \leq S_{j+1}}$ has the law of a rescaled and translated version of $B'$ above, and so its graph is non-removable with probability at least $q$, independently of the rest of $W$. Let $r = 1-q < 1$. Therefore for $N \in \N$ the probability that the graph of $(W_t)_{0 \leq t \leq S_N}$ is removable is no greater than~$r^N$, since this is the union of the smaller graphs, and if any of them are not removable, we can see that the larger graph is not either.

Therefore, fix $\ee > 0$ and choose $N$ such that $r^N < \frac\ee2$. Choose $a > 0$ such that $\P[ S_N \geq 1] < \tfrac\ee2$. If $S_N \leq 1$ and the graph, $\GG'$ of $(W_t)_{0 \leq t \leq S_N}$ is not removable, then clearly the graph $\GG$ of $(W_t)_{0 \leq t \leq 1}$ is not removable. Therefore, by a union bound
\[\P[\text{$\GG$ is not removable}] \leq \P[\text{$\GG'$ is not removable}]  + \P[S_N \leq 1] \geq 1 - \ee.\]

$(W_t)_{0 \leq t \leq 1}$ has the law of the Brownian motion on $[0,1]$ in the theorem statement, and $\ee$ was arbitrary, so it follows that the graph of a standard Brownian motion on a finite interval is a.s.\ not removable for $W^{1,p}$.
\end{proof}

\section{\texorpdfstring{Removability for $p = \infty$}{Removability}}\label{sec:bdd}
\subsection{Hausdorff dimension and measure}
First we review the notions of Hausdorff dimension and measure (see e.g.\ \cite[\S4]{mp2010} for further details). Let $E \subseteq \R^d$ and let $\n{E} = \diam E$. For $\aal \geq 0$, we define
\begin{equation}
    \cH^\aal_\dd(E) = \inf\left\{\sum_{i = 1}^\infty \n{E_i}^\aal \colon E \subseteq \bigcup_i E_i \text{  and  } \n{E_i} \leq \dd \; \forall i\right\}, \qquad \cH^\aal(E) = \lim_{\dd \to 0} \cH^\aal_\dd(E).
\end{equation}
$\cH^\aal(E)$ is referred to as the \emph{$\aal$-Hausdorff measure} of $E$. The \emph{Hausdorff dimension} of $E$ is defined to be 
\begin{equation}
    \dim_\cH(E) = \inf\{\aal \geq 0 \colon \cH^\aal(E) = 0\}.
\end{equation}
As discussed in \cite{mp2010}, the zero set $\cZ = \GG \cap \{y = 0\}$ of a one-dimensional Brownian motion a.s.\ has $\dim_\cH(\cZ) = \tfrac12$ and $\cH^{1/2}(\cZ) = 0$. The same statement holds if we instead let $\cZ$ be the intersection of $\GG$ with any horizontal line $\ell$ (conditional on $\GG \cap \ell \neq \varnothing$). This will be crucial to the first lemma below.

\subsection{Proof of Theorem~\ref{thm:bdd}}
We will now show that the graph $\GG$ of a standard Brownian motion on $[0,1]$ is a.s.\ $\woi$--removable, thus proving Theorem~\ref{thm:bdd}. Our proof of the theorem is divided into Lemma~\ref{lem:zero_cover} and Corollary~\ref{cor:short_paths}, which show that~$\GG$ a.s.\ has a certain property, and Lemma~\ref{lem:removable}, which shows that any compact set satisfying this property is $\woi$--removable.

\begin{lemma}\label{lem:zero_cover}
Let $(B_t)_{0 \leq t \leq 1}$ be a standard one-dimensional Brownian motion started from $0$ and let $\GG \subseteq \C$ be its graph. Let $\ell_y$ be the  horizontal line at height $y$.   Let $F_y$ be the event that there exists $c_y > 0$ (random) such that for every $k \in \N$, there exists a finite collection $\mathbf{U}_k$ of disjoint open\footnote{We will also allow intervals of the form $[0, b_j)$ and $(a_j, 1]$, which are open in $[0,1]$. Indeed, when $y = 0$, since $B_0 = 0$, we will always need some interval of the form $[0, b)$ to cover $\cZ$. We will occasionally not mention this and pretend that each $U$ can always be written as $U = (a,b)$ when in fact this may not be the case. We will point this out explicitly when necessary.} intervals $U_{kj} = (a_{kj}, b_{kj})$ covering $\cZ \coloneqq \pi(\GG \cap \ell_y)$, where $\pi$ denotes projection onto the $x$-axis, and satisfying the following conditions:
\begin{enumerate}[(a)]
    \item\label{it:a_cond} $\n{U_{kj}} < 1/k$ for each $j,k$.
    \item\label{it:b_cond} $\sum_{j} \n{U_{kj}}^{1/2} < 1/k$ for all $k$.
    \item\label{it:c_cond} $S_k \coloneqq \sum_j M_{kj} < c_y k^{-1/8}$ for all $k$, where for each $k,j$,
    \begin{align*}
    a_{kj}^* = \inf\left(\cZ \cap {U}_{kj}\right),\quad
    b_{kj}^* = \sup\left(\cZ \cap {U}_{kj}\right),\quad
    M_{kj} = \sup\{\n{B_t - y} \colon a_{kj}^* \leq t \leq b_{kj}^*\},
\end{align*}
if $\cZ$ intersects ${U}_{kj}$, and $M_{kj}$ is defined to be $0$ otherwise (see Figure~\ref{fig:uj}).
\end{enumerate}
Then $\P[F_y] = 1$ and, moreover, $F_y$ holds for a.e.\ $y \in \R$ a.s.
\end{lemma}
\begin{proof}
\emph{Step 1. Setup, measurability and conclusion.\\}
We will show a slightly stronger condition holds by imposing certain conditions on the constants and families of intervals to ensure the event we are looking at is measurable. In particular, we will only consider $\N$-valued constants $c_y$ and finite collections of disjoint open intervals \emph{with dyadic endpoints}. We define an open interval with dyadic endpoints to be any interval of the form $U = (a,b)$ where $a, b \in \{k2^{-n} \colon n \geq 0, k \in \{0,1,\dots, 2^n\}\}$, meaning they are dyadic rationals in $[0,1]$. We also include intervals of the form $[0, b)$ and $(a, 1]$ where $a,b$ are dyadic, and in a slight abuse of terminology refer to them as open intervals since they are open as subsets of the space $[0,1]$. Let $\cU$ denote the collection of all finite families $\mathbf{U}$ of disjoint open intervals with dyadic endpoints. The set $\cU$ is countable and can be ordered in some natural way, for example by looking first at collections $\mathbf{U}$ where each interval has endpoints in $2^{-n}\Z$ for $n = 0$, then $n \in \{0,1\}$, and so on.

\begin{figure}[t]
\includegraphics[scale=0.8]{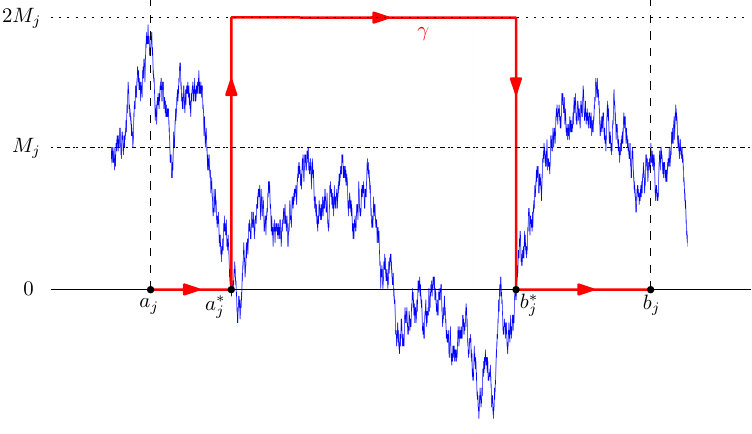}
\centering
\caption{Here we depict (in the case $y = 0$) an interval $U_j = (a_j, b_j) \in \cU$ where $a_j^*, b_j^*$ are defined as in the lemma statement (with $k$ suppressed in the notation). Also shown is part of the path $\gg$ constructed in Lemma~\ref{cor:short_paths} which intersects $\GG$ at most finitely many times.} 
\label{fig:uj}
\end{figure}

Let $(W, \cW, \P)$ be the set of continuous functions on $[0,1]$, let $\cW$ be the $\ss$-algebra generated by the coordinate projections, or equivalently the Borel $\ss$-algebra induced by the uniform norm $\inn{\cdot}$, and let $\P$ denote Wiener measure. 
Given constants $c,k \in \N$ and a family $\mathbf{U} = \{U_1, \dots, U_m\} \in \cU$, we define $E(c,k,\mathbf{U})$ to be the set of $(y,f) \in \R \times W$ where the family $\mathbf{U}$ covers $\cZ$, and where the conditions~\eqref{it:a_cond}, \eqref{it:b_cond} and~\eqref{it:c_cond} from the lemma statement hold, with $c$ in place of $c_y$, $\mathbf{U}$ in place of $\mathbf{U}_k$ and $f$ in place of $B$.
Note that~\eqref{it:a_cond} and~\eqref{it:b_cond} holding depends only on $k$ and $\mathbf{U}$ and not $y$ or $f$. We make the following claim and postpone its proof until after the proof of the lemma.

\begin{claim}\label{cl:mbl}
    The set $E(c,k,\mathbf{U})$ is measurable with respect to $\cB(\R) \otimes \cW$ where $\cB(\R)$ is the Borel $\ss$-algebra on $\R$. 
\end{claim}

Define
\[ E = \bigcup_{c \in \N} \bigcap_{k \in \N} \bigcup_{\mathbf{U} \in \cU} E(c,k,\mathbf{U}),\]
which, by the claim, is in $\cB(\R) \otimes \cW$. For a fixed $y \in \R$, define the event $E_y = \{f \in W \colon (y,f) \in E\}$, which is in $\cW$. Note that if $E_y$ holds, then so must $F_y$, the event described in the lemma statement. We will show in a moment that $\P[E_y] = 1$ but let us first show why this will complete the proof. In this case, $\P[E_y^c] = 0$ for every fixed $y \in \R$, so we have, using the measurability of the event $E$ and Fubini's theorem
\begin{align*}
    \E\!\left[\cL^1(\{y \colon (y,B) \in E^c\})\right] = \int_W \left(\int_\R \one_{(y,B) \in {E^c}}d\cL^1(y)\right)d\P[B] =  \int_\R \P[{E_y^c}]d\cL^1(y) = 0.
\end{align*}
Therefore $\cL^1(\{y \colon (y, B) \in E^c\}) = 0$ a.s., which shows that $E_y$ (and hence also $F_y$) holds for a.e.\ $y$ a.s., completing the proof.

\emph{Step 2. Showing $E_y$ holds a.s.\ for a fixed $y$.\\}
Assume first that $y = 0$. We will explain why this implies the general case at the end of the proof. 
Defining $E_y(c,k,\mathbf{U}) = \{f \in W \colon (y,f) \in E(c,k,\mathbf{U})\}$, we see that $E_y$ can be written as 
\[ E_y = \bigcup_{c \in \N} \bigcap_{k \in \N} \bigcup_{\mathbf{U} \in \cU} E_y(c,k,\mathbf{U}).\]
We will bound the probability of $\cup_\mathbf{U} E_y(c,k,\mathbf{U})$ in terms of $k$ and $c$, and show it is close to $1$. This will allow us to show that the probability of $\cap_{k} \cup_{\mathbf{U}} E_y(c,k,\mathbf{U})$ goes to $1$ as $c \to \infty$, thus showing that $\P[E_y] = 1$.

Suppose $c,k$ are fixed. Assume $\cH^{1/2}(\cZ) = 0$, which as mentioned above holds a.s. In this case, by the definition of Hausdorff measure, for any $k \in \N$ we can find an open (in $[0,1]$) cover $(V_j)$ of $\cZ$ such that $\n{V_j} < 1/(2k)$ for each $j$ and
\[ \sum_{j} \n{V_j}^{1/2} < \frac{1}{k\sqrt{2}}.\]
By the compactness of $\cZ$, we can assume this cover is finite. We may also assume without loss of generality that the $V_j = (c_j, d_j)$ are intervals with the usual caveat that if $c_j = 0$ or $d_j = 1$ then the interval may be closed at this endpoint.

Next we show that we may assume that the $V_j$ are disjoint. Suppose two of these intervals, say $V_1 = (c_1, d_1)$ and $V_2 = (c_2, d_2)$, have non-empty intersection. If one is contained in the other, we may simply remove it, so we can assume without loss of generality that $c_1 < c_2 < d_1 < d_2$. Since $\cZ$ has zero Lebesgue measure, $(c_2, d_1) \setminus \cZ$ is non-empty and open, and therefore contains a non-empty interval $(e_1, e_2)$. We can now redefine $V_1 = (c_1, d_1')$ and $V_2 = (c_2', d_2)$ where $e_1 < d_1' < c_2' < e_2$. Then $V_1$ and $V_2$ become smaller and the $V_j$ still cover $\cZ$. Repeating this process for every pair of intersecting intervals allows us to assume the $V_j$ are indeed disjoint. By a similar procedure, we can in fact conclude that their closures $\overline{V_j}$ are also disjoint. Then, for each $j$, we can find an open interval with dyadic endpoints $U_j$ containing $V_j$ such that $\n{U_j} \leq 2\n{V_j}$. Since the closures of the $V_j$ are disjoint, we can make this choice in such a way that the $U_j$ are also disjoint. It follows that there always exists $\mathbf{U} \in \cU$ covering $\cZ$ such that~\eqref{it:a_cond} and~\eqref{it:b_cond} in the lemma hold. Let $C(\mathbf{U}) \equiv C$ be the event that $\cZ$ is covered by $\mathbf{U}$ and that $\mathbf{U}$ satisfies~\eqref{it:a_cond} and~\eqref{it:b_cond} in the lemma (this final requirement does not depend on $B$ but including it simplifies notation). We have just shown that a.s., there exists $\mathbf{U} \in \cU$ such that $C(\mathbf{U})$ holds.

Fix $\mathbf{U} = \{U_1, \dots, U_m\} \in \cU$. Note that the event $C(\mathbf{U})$ depends on $B$ only through $\cZ$, the zero set of $B$, which is a random variable taking values in the set of closed subsets of $[0,1]$. 
Almost surely, $[0,1] \setminus \cZ$ is a countable collection of open intervals.  As is explained in \cite[Chapter XII]{revuz_yor}, the It\^o excursion decomposition of $B$ realizes it as a Poisson point process of excursions from $0$. This representation shows that conditional on $\cZ$, on each interval in the complement of $\cZ$, $B$ has the law of a (signed) Brownian excursion of a specified length.

Next we estimate the conditional probability $\P[E_y(c,k,\mathbf{U}) | \cZ]$, where $\cZ$ and $\mathbf{U}$ are such that $C$ holds. Fix one such $\cZ$, and let $\wt\P$ the law of $B$ conditional on $\cZ$.  We let $a_j^* = \inf(\cZ \cap U_j)$ and $b_j^* = \sup(\cZ \cap U_j)$.  For each $j$ such that $\cZ \cap U_j \neq \emptyset$, the Brownian motion on $(a_j^*, b_j^*)$ consists of countably many (signed) Brownian excursions of lengths $(\ell_j^n)_{n \in \N}$ where $\sum_n \ell_j^n = b_j^* - a_j^* \leq \n{U_j} =: \ell_j$. Let $M_j^n$ be the maximum of each of (the absolute value of) these excursions, so that $M_j = \sup_n M_j^n$, where $M_j$ is defined analogously to $M_{jk}$ in the lemma statement.

\begin{claim}\label{cl:be}
For $x \geq \sqrt{4\ell_j/e}$, we have $\wt\P[M_j \geq x] \leq 2\exp\left(-\frac{x^2}{4\ell_j}\right).$ 
\end{claim}
We postpone the proof until after completing the proof of the lemma. Assuming the claim, and introducing constants $d_1, d_2 \in \N$, we have
\begin{align*}
    \wt\E[M_j] &= \int_0^\infty \wt\P[M_j \geq x] dx \leq \sqrt{\frac{4\ell_j}{e}} + \int_0^\infty 2\exp\left(-\frac{x^2}{4\ell_j}\right)dx = \left( \sqrt{\frac{4}{e}} + \sqrt{4\pi}\right)\sqrt{\ell_j} \leq d_1 \n{U_j}^{1/2},\\
\var_{\wt\P}[M_j] &\leq \wt\E[M_j^2] = \int_0^\infty \wt\P[M_j \geq x^{1/2}]dx \leq \frac{4\ell_j}{e} + \int_0^\infty 2\exp\left(-\frac{x}{4\ell_j}\right)dx \leq \left( \frac{4}{e} + 8\right)\ell_j \leq d_2 \n{U_j}.
\end{align*}
Define $S_k = \sum_{j=1}^m M_j$. Then we can conclude, using~\eqref{it:a_cond} and~\eqref{it:b_cond} in the lemma statement, that
\begin{align*}
    \wt\E[S_k] &\leq d_1 \sum_j \n{U_j}^{1/2} \leq d_1 k\nv,\\
    \var_{\wt\P}[S_k] &\leq d_2\sum_j\n{U_j} \leq d_2k^{-3/2}.
\end{align*}
By Chebyshev's inequality,
\[\wt\P\!\left[S_k \geq c k^{-1/8} + d_1 k\nv\right] \leq \frac{\var_{\wt\P}[S_k]}{(c k^{-1/8})^2} \leq \frac{d_2k^{-3/2}}{c^2 k^{-1/4}} = \frac{d_2}{c^2} \cdot k^{-5/4}.\]
Let $c' = c + d_1$ and recall that if $C$ is true and $S_k < ck^{-1/8} + d_1k\nv$, then $E_y(c',k,\mathbf{U})$ holds  since $k\nv \leq k^{-1/8}$. By the definition of $\wt\P$, this means that for any $\mathbf{U} \in \cU$ we have
\[ \P[E_y(c',k,\mathbf{U}) | \cZ] \geq \left(1 - \frac{d_2}{c^2}k^{-5/4}\right)\one_{C(\mathbf{U})}.\]
The LHS is a conditional probability, so this inequality should be understood in the almost sure sense. Since $\cU$ is countable, the following also holds a.s.: 
\[ \P\!\left[\bigcup_{\mathbf{U} \in \cU} E_y(c',k,\mathbf{U})\middle|\cZ\right] \geq \sup_{\mathbf{U} \in \cU} \left[\left(1 - \frac{d_2}{c^2}k^{-5/4}\right)\one_{C(\mathbf{U})}\right] = 1 - \frac{d_2}{c^2}k^{-5/4},\]
where the last inequality holds since we have shown earlier in the proof that a.s., $C(\mathbf{U})$ holds for some $\mathbf{U} \in \cU$. We obtain
\[\P\!\left[\bigcup_{\mathbf{U} \in \cU} E_y(c',k,\mathbf{U})\right] = \E\!\left[\P\!\left[\bigcup_{\mathbf{U} \in \cU} E_y(c',k,\mathbf{U})\middle|\cZ\right]\right] \geq 1 - \frac{d_2}{c^2}k^{-5/4}.\]
By a union bound on the complement, and since $k^{-5/4}$ is summable, we conclude that
\[ \P\!\left[\bigcap_{k \in \N}\bigcup_{\mathbf{U} \in \cU} E_y(c',k,\mathbf{U})\right] \geq 1 - g(c'),\]
where $g(c') \to 0$ as $c' \to \infty$. Finally, it follows that $\P[E_y] = 1$ for $y = 0$, as claimed.

Suppose $y \neq 0$. Let $B'$ be a standard Brownian motion run until time $\tau_y + 1$ where $\tau_y = \inf\{t \geq 0 \colon B'_t = y\}$. Let $B''_t = \wt{B}_{\tau_y + t} - y$ so that $B''$ and $B$ have the same law. Therefore a.s.\ there exists a constant $c_y$ and families of intervals $\mathbf{U}_k'$ such that we can cover the zero set of $B''$ by $\mathbf{U}_k'$ and~\eqref{it:a_cond}, \eqref{it:b_cond} and~\eqref{it:c_cond} hold with $B''$ and $\mathbf{U}'$ replacing $B$ and $\mathbf{U}$ respectively. But the zero set of $B''$ corresponds to the level set $\{t \colon B'_t = y\}$. If we let $B$ be the restriction of $B'$ to $[0,1]$, then a.s., if we set $U_{kj} = U'_{kj} \cap [0,1]$ and let $\mathbf{U}_k$ be the collection of the sets $U_{kj}$ which are non-empty, then we see that $\mathbf{U}_k$ will cover $\cZ$ and satisfy the conditions of the lemma for each $k$. It follows that $\P[E_y] = 1$ for all $y$. Note that a special case of this is when $B$ does not hit $y$ before $t = 1$, in which case $\ell_y$ will not intersect $\GG$ at all; the statement in the lemma is then trivially satisfied with $\mathbf{U}_k$ being the empty family for all~$k$.
\end{proof}

We now return to the two claims made during the proof of the lemma.
\begin{proof}[Proof of Claim~\ref{cl:mbl}]
We will show the set $E(c, k, \mathbf{U})$ is open in $\R \times W$ where $W$ is equipped with the uniform norm. Suppose $(y,f) \in E(c,k,\mathbf{U})$, where $\mathbf{U} = \{U_1, \dots, U_m\}$ and $U_j = (a_j, b_j)$ for each~$j$, with the caveat that if $a_j = 0$ or $b_j = 1$, then this endpoint of the interval might be closed. Let $\cZ_y(f) = \{x \in [0,1] \colon f(x) = y\}$. The sets $\cZ_y(f)$ and $[0,1] \setminus \cup_j U_j$ are closed and disjoint, so the continuous function $x \mapsto \n{f(x) - y}$ is bounded below by some $\dd > 0$ on $[0,1] \setminus \cup_j U_j$. Therefore if $\n{y_1 - y} < \dd/2$ and $\inn{f_1 - f} < \dd/2$, we have that $\mathbf{U}$ still covers $\cZ_{y_1}(f_1)$. 

Next we want to show that condition~\eqref{it:c_cond} still holds if we change $y$ and $f$ slightly. This is made more difficult by the fact that the $a_j^*$ and $b_j^*$ may change, which necessitates the following argument. Suppose $S_k + \eta < ck^{-1/8}$ for some $\eta > 0$. For each $j$, define 
\begin{align*}
    U'_j = U_j \cap \{x  \colon \n{f(x) - y} \geq M_j + \tfrac{\eta}{2m}\},\quad
    a_j' = \sup \big(U_j' \cap [a_j, a_j^*]\big),\quad
    b_j' = \inf \big(U_j' \cap [b_j^*, b_j]\big).
\end{align*}
For a fixed $j$ there are now a few possibilities. Firstly, it may be that $f(x) \neq y$ on all of $U_j$, in which case we have defined $M_{kj} = 0$. If this occurs, then since $\cZ_y(f)$ is closed and the $U_j$ are disjoint, we also must have that $f(a_j), f(b_j) \neq y$. Therefore there exists $\ee_j > 0$ such that $\n{f(x) - y} > \ee_j$ on~$U_j$. Therefore if $\n{y_1 - y} < \ee_j/2$ and $\inn{f_1 - f} < \ee_j/2$ then $\cZ_{y_1}(f_1)$ will still not intersect $U_j$ so the value of $M_j$ corresponding to $y_1$ and $f_1$, which we denote by $M_j(y_1, f_1)$ will still be $0$.

Secondly, if $\cZ_y(f) \cap U_j \neq \varnothing$ and $U_j'$ is empty then, for any $\ee > 0$, if $\n{y_1 - y} < \ee/2$ and $\inn{f_1 - f} < \ee/2$ we will have that $M_j(y_1, f_1) \leq M_j + \eta/2m + \ee$, since in this case even if $a_j^*$ and $b_j^*$ change, they will always stay in $\overline{U}_j$, where we have control over $f$.

Finally, suppose $\cZ_y(f) \cap U_j \neq \varnothing$ and $U_j'$ is not empty. Suppose $U_j' \cap [a_j, a_j^*]$ is not empty and $U_j' \cap [b_j^*, b_j]$ is empty. By the definition of $M_j$ we will have that $U_j'$ and $(a_j^*, b_j^*)$ are disjoint, and since $f(a_j^*) = f(b_j^*) = y$ we see that we will actually have $a_j \leq a_j' < a_j^*$. In this case, by the definition of $a_j^*$ as the first time in $U_j$ that $f(x) = y$, we have that $\n{f(x) - y}$ must be bounded below by some $\ee_j > 0$ on $[a_j, a_j']$.  Therefore if $\n{y_1 - y} < \ee_j/2$ and $\inn{f_1 - f} < \ee_j/2$ then, since $f_1(x) \neq y_1$ on $[a_j, a_j']$, this interval will not intersect $[a_j^*(y_1, f_1), b_j^*(y_1, f_1)]$, where $a_j^*(y_1, f_1)$ is the new value of $a_j^*$ corresponding to $y_1, f_1$ (similarly for $b_j^*$) . Therefore, we will have that $M_j(y_1, f_1) \leq M_j + \eta/2m + \ee_j$. Note that the value of $b_j^*(y_1, f_1)$ doesn't matter since we have control of $f$ on $[b_j^*, b_j]$. The case where $U_j' \cap [b_j^*, b_j]$ is not empty is handled in a similar way.

To sum up, if we choose $\ee$ such that it is less than all the $\ee_j$ for all $j$ for which the first and third cases apply, and also that $m\ee < \eta/2$, then we finally have that $M_j(y_1, f_1) \leq M_j + \eta/2m + \ee$ for every $j$, so $S_k(y_1, f_1) \leq S_k + \eta/2 + m\ee < ck^{-1/8}$.

In conclusion, if $(y,f) \in E(c,k,\mathbf{U})$ then we can choose $\ee > 0$ small enough that if $\n{y_1 - y} < \ee/2$ and $\inn{f_1 - f} < \ee/2$ we will have that $\mathbf{U}$ still covers $\cZ_{y_1}(f_1)$ and that~\eqref{it:c_cond} still holds for $y_1$ and $f_1$, meaning that $(y_1, f_1) \in E(c,k,\mathbf{U})$. Therefore $E(c,k,\mathbf{U})$ is open in $\R \times W$, and hence is measurable, as claimed.
\end{proof}

\begin{proof}[Proof of Claim~\ref{cl:be}]
Note that, by a union bound,
\[ \wt\P[ M_j \geq x] = \wt\P\!\left[\sup_n M_j^n \geq x\right] \leq \sum_n \wt\P[M_j^n \geq x].\]
Recall that $M_j^n$ is the maximum process of a Brownian excursion of length $\ell_j^n$. It was shown in \cite{kennedy, chung} that $M_j^n$ has the same distribution as 
\[ \sup_{0 \leq t \leq \ell_j^n} B^0_t - \inf_{0 \leq t \leq \ell_j^n} B^0_t\]
where $B^0$ is a Brownian bridge of length $\ell_j^n$. For a Brownian bridge of length $\ell_j^n$ we have that $\wt\P[\sup_{0 \leq t \leq \ell_j^n} B^0_t \geq x] = \exp(-2x^2/\ell_j^n)$. Combining these results and using a union bound we have that
\begin{equation}
\label{eqn:max_bound}
\wt\P[M_j^n \geq x] \leq 2\,\wt\P\!\left[\max_t B_t^0 \geq \frac{x}{2}\right] = 2 \exp\left(- \frac{x^2}{2\ell_j^n}\right) = 2\exp\left(- \frac{(x/\sqrt{2\ell_j})^2}{(\ell_j^n/\ell_j)}\right).
\end{equation}
Set $h_n = \ell^n_j/\ell_n \in [0,1]$. Then we will show that
\begin{equation}
\label{eqn:w_h_n_ineq}
\frac{w^2}{h_n} \geq \frac{w^2}{2} - \log(h_n) \quad\text{for}\quad w \geq \sqrt{2/e}.
\end{equation}
Rearranging, we see that~\eqref{eqn:w_h_n_ineq} is equivalent to
\begin{equation}
\label{eqn:w_h_n_ineq_equiv}
w^2\left(1 - \frac{h_n}{2}\right) \geq -h_n \log h_n.
\end{equation}
The right hand side of~\eqref{eqn:w_h_n_ineq_equiv} is bounded from above by $1/e$ since $h_n \in [0,1]$ and the left hand side of~\eqref{eqn:w_h_n_ineq_equiv} is bounded below by $\tfrac12 w^2$.  Therefore~\eqref{eqn:w_h_n_ineq_equiv} holds whenever $w \geq \sqrt{2/e}$, which proves~\eqref{eqn:w_h_n_ineq}.

Applying~\eqref{eqn:w_h_n_ineq} with $w = x/\sqrt{2\ell_j}$ we have for $x \geq \sqrt{4\ell_j/e}$ that
\[ \wt\P[M_j^n \geq x] \leq 2\exp\left(-\frac{x^2}{4\ell_j} + \log(\ell_j^n/\ell_j)\right) = 2\exp\left(-\frac{x^2}{4\ell_j}\right)\cdot\frac{\ell_j^n}{\ell_j}.\]
Summing over all $n$ completes the proof.
\end{proof}

\begin{cor}\label{cor:short_paths}
    Almost surely, $\GG$ satisfies the following property. For any $z,w \in \C$ and any $\ee > 0$, there exists a path $\gg$ from $z$ to $w$ which consists of a finite union of line segments, intersects $\GG$ finitely many times, and has length at most $\sqrt{2}\nn{z - w} + \ee$.
\end{cor}
\begin{proof}
Fix $\ee > 0$. Suppose first that $z = (u, y')$, $w = (v, y')$ for some $u,v,y' \in \R$. By Lemma~\ref{lem:zero_cover}, there exists $y \in \R$ with $y > y'$ and $\n{y - y'} < \ee/4$ such that $F_y$ holds. Choose $k \in \N$ such that $c_y k^{-1/8} < \ee/8$ and let $\mathbf{U}_k = \{ U_1, \dots, U_m \}$ be a family of intervals as described in the lemma. If any $U_j$ does not intersect $\cZ$ we can remove it, so we can assume that $a_j^* = \inf(\cZ \cap U_j)$ and $b_j^* = \sup(\cZ \cap U_j)$ lie in $U_j$ for each $j$. We can also assume that the $U_j$ are in increasing order, so that $a_1^* \leq b_1^* \leq a_2^* \leq \cdots \leq b_m^*$. Let $E_j = [a_j^*, b_j^*]$ for each $j$.

We define a path $\gg$ as follows (see also Figure~\ref{fig:uj}). Suppose first $u \in E_{j_0}$. From $z$, move vertically to upwards until $\gg$ hits the horizontal line at height $y + 2M_j$. Then move to the right until we are at $(b_{j_0}^*, y + 2M_j)$. Then move downwards to $(b_{j_0}^*, y)$.\footnote{if $u = b_{j_0}^*$ then $\gg$ may immediately retrace itself in the opposite direction after hitting $(b_{j_0}^*, y + 2M_j)$. This is not an issue, but if we want $\gg$ to be simple we can just remove the part that is traced twice from $\gg$.} Note that the only time $\gg$ intersects $\GG$ is possibly during the first vertical segment and at $(b_{j_0}^*, y)$. We then move to the right until we hit $a_{j_0 + 1}^*$. For each subsequent $j \geq j_0$ we perform a similar procedure. Starting at $(a_j^*, y)$ move from $(a_j^*, y)$ to $(a_j^*, y + 2M_j)$ to $(b_j^*, y + 2M_j)$ to $(b_j^*, y)$ and finally on to $(a_{j+1}^*, y)$. If we reach a $j_1$ such that $v \in E_{j_1}$, we perform a reversed version of the procedure we used to start $\gg$. If $v$ is not in any of the $E_j$, then we simply stop $\gg$ when it hits $(v, y)$, at which point we move back down to $w = (v, y')$. If $u$ is instead not in any of the $E_j$, we start by moving upwards from $z = (u,y')$ to $(u, y)$, then moving right until we hit one of the $a_j^*$, after which we perform the procedure described above. Note that during this process $\gg$ intersects $\GG$ at most finitely many times.

The outcome of this is that the total length $\ell(\gg)$ of $\gg$ is bounded by $2(y - y') + \sum_{j=1}^m 4M_j + (v-u)$, where the first term and the sum correspond to vertical distance traveled and the last term to horizontal distance. Therefore
\[\ell(\gg) \leq 2(y - y') + \sum_{j=1}^m 4M_j + (v-u) < \frac{\ee}{2} + 4S_j + \n{v-u} < \nn{z - w} + \ee.\]

Now let $z = (z_1, z_2), w = (w_1, w_2)$ be general. Set $z' = (z_1, w_2)$ and let $\gg$ be the concatenation of the vertical line segment from $z$ to $z'$, and the path $\gg'$ that we can construct from $z'$ to $w$ by the above. In this case,
\[ \ell(\gg) \leq \n{z_2 - z_1} + \nn{z' - w} + \ee \leq \sqrt{2}\nn{z - w} + \ee,\]
as required.  
\end{proof}

We have shown that $\GG$ a.s.\ satisfies the preconditions of the following lemma, which will then complete the proof of Theorem~\ref{thm:bdd}.

We briefly mention that our proof actually slightly extends the result of \cite[Theorem 6]{tecu} in the deterministic case. Indeed, suppose $\GG'$ is the graph of a deterministic function $g \colon [0,1] \to \R$ that is $1/2$-H\"older continuous and $\cH^{3/2}(\GG') = 0$. In this case, if we define $\cZ = \pi(\GG' \cap \ell)$ then $\cH^{1/2}(\cZ) = 0$ for a.e.\ horizontal line $\ell$ (see e.g. \cite[Theorem 7.7]{mattila}). It follows that on such lines $\ell$, for each $k \in \N$ we can cover $\cZ$ by a family of intervals $\mathbf{U}$ satisfying~\eqref{it:a_cond} and~\eqref{it:b_cond} in Lemma~\ref{lem:zero_cover}. Let $U_j = (a_j, b_j)$ for each $U_j \in \mathbf{U}$ and, using the same notation as in the proof of Lemma~\ref{lem:zero_cover}, define $M_j = \sup\{g(t) - y \colon a_j^* \leq t \leq b_j^*\}$. Then $M_j \leq c\n{U_j}^{1/2}$ for some fixed constant $c > 0$ depending on $g$ using H\"older continuity (which replaces the probabilistic argument given above). Then we immediately get that $S_k \leq ck\nv$ for all $k$ (using~\eqref{it:b_cond}) which is stronger than the bound shown in Lemma~\ref{lem:zero_cover}. Therefore, we can repeat the proof of the preceding corollary to show that this deterministic graph satisfies the property described in the corollary statement. This means that the following lemma also applies in this case, showing that~$\GG'$ is $\woi$--removable.

\begin{lemma}\label{lem:removable}
Let $K \subseteq \C$ be compact and suppose there exists $C > 0$ such that for any $z,w \in \C$ and any $\ee > 0$, there exists a path $\gg$ from $z$ to $w$ which consists of a finite union of line segments, intersects $K$ finitely many times, and has length at most $C\nn{z - w} + \ee$. Then $K$ is $W^{1,\infty}$--removable.
\end{lemma}
\begin{proof}
To show $\GG$ is $W^{1,\infty}$--removable, we need to show that any continuous $f \in W^{1,\infty}(\C \setminus K)$ is in $W^{1,\infty}(\C)$. It can be shown (see e.g. \cite[Theorem 4.1]{heinonen}) that any such $f$ is locally $L$-Lipschitz on $\C \setminus K$ where $L = \inn{\nabla f}$. That is, every $x \in \C \setminus K$ has a neighborhood on which $f$ is $L$-Lipschitz. It follows that $\n{f(z) - f(w)} \leq L\nn{z - w}$ whenever the line connecting $z$ and $w$ does not intersect $K$. Using this fact and continuity, we can show that if a path $\gg$ from $z$ to $w$ is a finite union of line segments that intersects $\GG$ finitely many times, then $
\n{f(z) - f(w)} \leq L\ell(\gg)$.

By the hypotheses of the lemma, for any $z,w \in \C$ and $\ee > 0$ there exists such a path $\gg$ connecting $z$ and $w$ with $\ell(\gg) \leq C\nn{z-w} + \ee$, meaning that 
\[ \n{f(z) - f(w)} \leq CL\nn{z - w} + L\ee \quad\text{for all}\quad \ee > 0,\]
and hence that $f$ is Lipschitz on $\C$ with Lipschitz constant no larger than $CL$. Since $f$ is in $\woi(\C\setminus K)$ it is bounded, so using \cite[Theorem 4.1]{heinonen} again shows that $f \in \woi(\C)$, completing the proof.
\end{proof}

\bibliographystyle{abbrv}
\bibliography{references}

\end{document}